\documentclass[reqno]{amsart}
\usepackage{amssymb,amsthm,amsfonts,amstext,amsmath}
\usepackage{newcent}
\usepackage{changes}
\definechangesauthor[name=Guilherme, color=blue]{Gui}

\usepackage{helvet}          
\usepackage{courier}         

\usepackage{graphicx}
\usepackage{enumerate}
\numberwithin{equation}{section}
\usepackage{color}
\usepackage{float}
\usepackage{mathrsfs}
\usepackage[T1]{fontenc}

\usepackage{float}
\usepackage{dsfont}
\usepackage[colorlinks=true,linkcolor=blue,citecolor=magenta]{hyperref}

\usepackage{tikz}
\usetikzlibrary{arrows,decorations.pathmorphing}
\usetikzlibrary{calc, shapes, backgrounds,arrows.meta,patterns,fadings}
\usetikzlibrary{decorations.markings,positioning}

\def\centerarc[#1](#2)(#3:#4:#5){\draw[#1] ($(#2)+({#5*cos(#3)},{#5*sin(#3)})$) arc (#3:#4:#5);}
\tikzset{
  big arrow 3/.style={
    decoration={markings,mark=at position 0.7 with {\arrow[scale=0.7,#1]{>}}},
    postaction={decorate}},
  big arrow/.default=black}
  
\def\centerarc[#1](#2)(#3:#4:#5){\draw[#1] ($(#2)+({#5*cos(#3)},{#5*sin(#3)})$) arc (#3:#4:#5);}

\newtheorem{theorem}{Theorem}[section]
\newtheorem{lemma}[theorem]{Lemma}
\newtheorem{proposition}[theorem]{Proposition}

\newtheorem{remark}[theorem]{Remark}
\newtheorem{definition}[theorem]{Definition}
\newtheorem{conjecture}[theorem]{Conjecture}

\newcommand{\one}{\mathds{1}}
\newcommand{\mc}[1]{{\mathcal #1}}

\newcommand{\bb}[1]{{\mathbb #1}}

\newcommand{\p}{\partial}

\newcommand{\dd}{\displaystyle}

\newcommand{\ignore}[1]{}

\keywords{Reinforced random walk, ant random walk, directed edges, random walk on graphs}

\begin{document}

\title[The Directed Edge Reinforced Random Walk]{The Directed Edge Reinforced Random Walk:\\ The Ant Mill Phenomenon}

\author[D. Erhard]{Dirk Erhard}
\address{UFBA\\
 Instituto de Matem\'atica, Campus de Ondina, Av. Adhemar de Barros, S/N. CEP 40170-110\\
Salvador, Brazil}
\curraddr{}
\email{erharddirk@gmail.com}
\thanks{}

\author[T. Franco]{Tertuliano Franco}
\address{UFBA\\
 Instituto de Matem\'atica, Campus de Ondina, Av. Adhemar de Barros, S/N. CEP 40170-110\\
Salvador, Brazil}
\curraddr{}
\email{tertu@ufba.br}
\thanks{}

\author[G. Reis]{Guilherme Reis}
\address{UFBA\\
 Instituto de Matem\'atica, Campus de Ondina, Av. Adhemar de Barros, S/N. CEP 40170-110\\
Salvador, Brazil}
\curraddr{}
\email{ghreis@impa.br }
\thanks{}

\subjclass[2010]{60K35, 60K37, 60G50}

\begin{abstract} 
We define here a \textit{directed edge reinforced random walk} on a connected locally finite graph.  As the name suggests, this walk keeps track of its past, and gives a bias towards directed edges previously crossed  proportional to the exponential of the number of crossings. The model is inspired by the so called \textit{Ant Mill phenomenon}, in which a group of army ants forms a continuously rotating circle until they die of exhaustion. For that reason we refer to the walk defined in this work as the \textit{Ant RW}.
Our main result justifies this name. Namely, we will show that on any finite graph which is not a tree, and on  $\bb Z^d$ with $d\geq 2$, the Ant RW almost surely gets eventually trapped into some directed circuit which will be followed forever. 
In the case of~$\bb Z$ we show that the Ant RW eventually escapes to infinity and satisfies a law of large number with a random limit which we explicitly identify.
\end{abstract}

\maketitle


\allowdisplaybreaks

\section{Introduction}
The Ant Mill is a phenomenon in which a group of blind army ants gets separated from their main group and, guided by pheromones, start to walk behind one another and in this way forms a circuit they follow until they die of exhaustion.  We refer the interested reader to the paper \cite{Delsuc}  a discussion of that phenomenon, and  to the video \cite{video} for an illustration.

In this work we investigate a model that probabilistically encodes the above phenomenon in the case of a single ant  on connected non-tree finite graphs and on $\bb Z^d$, $d\geq 2$. We then interpret the ant as a random walk with a bias towards already visited \textit{directed} edges. Here the bias increases with each crossing of a directed edge, and it decreases whenever an edge is crossed in the opposite direction. Put differently what counts is the ``net'' number of crossings.

Our model can be placed into the world of \textit{reinforced} random walks. To the best of our knowledge this notion goes back to~\cite{CopperDiaconis,Diaconis88, Pemantle1988}.  Since then, a large literature has been developed and reinforced random walks have become an active and challenging area of research. Among the most prominent models are the \textit{vertex} reinforced random walk~\cite{Pemantle1988, Pemantle1992}  and the \textit{edge} reinforced random walk~\cite{CopperDiaconis,Diaconis88}, where the bias is proportional to the number of times a certain vertex and edge respectively has been visited. 
One of the questions of interest in these models is concerned with localisation, i.e., will the random walk be eventually trapped in a finite region?
For the vertex reinforced random walk this is indeed the case as has been shown in numerous works with different stages of refinement~\cite{BRS2013,LimicTarres2010,PemantleVolkov1999,Tarres2004,Volkov2001}. 
Here, depending on the strength of the reinforcement and the underlying graph the walk may localise on two or more vertices. For the edge reinforced walk, similar results have been obtained. 

In~\cite{LimicTarres2010} it was for instance shown that if the sum of inverse of weights is finite  and under some further technical assumptions the walk eventually gets stuck on a single edge. We also mention  a model with a similar flavour and names as ours, namely, the
 \textit{directionally} reinforced random walk, which  was investigated in~\cite{Directionally2,Directionally1}. However, in that model the walker looses its memory after each change of direction, which makes the model fundamentally different to ours.

In the present work we aim at showing localisation as in the vertex or edge reinforced models. However, since the reinforcement is along directed edges  localisation on a single edge is not possible; jumping forth and back over the same edge neutralises the reinforcement. Instead we will show in our main result, Theorem~\ref{thm2.1}, localisation on circuits, which  justifies the Ant RW name for the walk: this theorem  states that on non-tree finite graphs  and on $\bb Z^d$ for $d\geq 2$, the Ant RW (Ant Random Walk) with probability one eventually gets trapped in a directed circuit which will be followed forever, similarly to the Ant Mill phenomenon mentioned at the beginning of this introduction.

What makes the Ant RW so challenging is that it is heavily non Markovian,  due to the fact that at each step the behaviour of the walker depends on its entire past. In the two previously  described models a feature that partially compensates that difficulty is monotonicity, i.e., the more often a vertex, respectively edge, is visited the more attractive it will become in the future. 

In our model this is not the case.  Indeed, if an edge $(x,y)$ was crossed as many times as the edge $(y,x)$ it is as if neither of the two were ever crossed, i.e., it is possible to ``kill'' a bias by crossing an edge in the reversed direction. Consequently, classical tools such as P\' olya Urn techniques, e.g., the Rubin construction in \cite{davis1990}, are not directly available. 

To partially compensate for that difficulty we, at least for the moment, work with a strong, i.e., exponential reinforcement. This then enables us to analyse the model in two steps. The first is completely deterministic and investigates the evolution of the environment, i.e., the field of crossing numbers induced by a fixed path. Having gained sufficient information on the environment we then use a renewal property of the dynamics to conclude the analysis. A relevant feature of the paper is the understanding of the environment;   we believe that this comprehension will be useful also for weaker reinforcement versions of the model. Finally, it is  worthy commenting that \textit{ant inspired algorithms} are in great development nowadays in Computer Science (see for instance \cite{Dorigo1,Dorigo2,Yang} and references therein), for which our result  may be  applicable.

\textbf{Organization of the paper.}
The article is organized as follows. In Section~\ref{s2}, the model is precisely defined, the main results are stated, and the main ideas are discussed. In Section~\ref{sec:3} we give the proof of Theorem~\ref{thm2.1} in the finite graph case and in Section~\ref{sec:Zd} we show  Theorem~\ref{thm2.1} in the case of $\bb Z^d$ with $d\geq 2$. Finally in the Appendix~\ref{sec:d=1} we provide the proof of Proposition~\ref{th:it_is_markovian}.
  
\section{Statements}\label{s2}
We define here the directed edge reinforced random walk, which will be  referred as  \textit{Ant RW} in the sequel, as a discrete time stochastic process on  some locally finite, connected, undirected graph  $G$ with
vertex set $V=V(G)$ and edge set $E=E(G)$.  
Given two vertices $v$ and $w$ we write $v\sim w$ if the pair $(v,w)$ forms an edge. We then define the stochastic process $(X_n)_{n\in \bb N}$  with state space $V$ by the
following transition rule. Fix a vertex $v$, and set $ X_0=v$.
For $n \geq  0$ and $\beta\in(0,\infty)$, we define
\begin{align} \label{eq:transitions}
\bb P(X_{n+1} = x | \mc G_n)\; =\; \frac{a_n(X_n, x)}{\displaystyle
	\sum_{y\sim X_n} a_n(X_n, y)}\,,
\end{align}
where $\mc G_n = \sigma(X_0, X_1, \ldots, X_n)$ is the $\sigma$-algebra generated by the walk up to time $n$. Here the \textit{weights} $a_n$ are given by
\begin{align*}
a_n(X_n,x)\;=\;\exp\big\{\beta \,c_n(X_n, x)\big\}\,\,,
\end{align*}
and the \textit{crossing numbers} $c_n(x,y)$ above are defined via
\begin{align} \label{eq:edge_weights}
c_n (x,y) \;=\; \sum_{k=0}^{n-1}\Big( \one\big[(X_k,X_{k+1})=(x,y)\big]-\one\big[(X_k,X_{k+1})=(y,x)\big]\Big)\,.
\end{align}
In plain  words, $c_n (x,y)$ is  the number of times that, up to time $n$, the walk has jumped from $x$ to $y$ minus the number of times it has jumped from $y$ to $x$. 
The parameter $\beta\in (0,\infty)$ represents the strength of the reinforcement. In the limiting case $\beta=0$ we recover the usual symmetric random walk, whereas in the other limiting case $\beta= \infty$ once the walk has crossed a certain edge $(x,y)$ from $x$ to $y$, it will always choose the same edge in the same direction once it returns to~$x$.

Our first observation reads as follows and shows that the behaviour of $X$ on $G=\bb Z$ is particularly simple. The proof follows from elementary observations and is provided in Appendix~\ref{sec:d=1}.
\begin{proposition}\label{th:it_is_markovian}
Let $G=\bb Z$ and assume $X_0=0$. Then the Ant RW $(X_n)_{n\geq 0}$ is a Markov chain with transition probabilities given by
\begin{equation}\label{eq:trans1d}
\begin{split}
\bb P\big(X_{n+1}\;=\;\pm 1|X_n=0\big)&\;=\;\frac{1}{2}\,,\\
\bb P\big(X_{n+1}=x+1|X_n=x\big)&\;=\;
\begin{cases} \dfrac{1}{1+e^{-\beta}},\quad \mbox{ if } x\geq 1,\vspace{4pt}\\
 \dfrac{e^{-\beta}}{1+e^{-\beta}},\quad \mbox{ if } x\leq -1, 
\end{cases}\\
\bb P\big(X_{n+1}=x-1|X_n=x\big)&\;=\;\begin{cases} \dfrac{e^{-\beta}}{1+e^{-\beta}},\quad \mbox{ if } x\geq 1,\vspace{4pt}\\ 
\dfrac{1}{1+e^{-\beta}},\quad \mbox{ if } x\leq -1. \end{cases}
\end{split}
\end{equation}
In particular, the Ant RW on $\bb Z$ is transient and satisfies the following law of large numbers:
\begin{align}\label{lln}
\lim_{n\to \infty} \frac{X_n}{n}\;=\; Y\quad \text{a.s.}
\end{align}
where $\bb P\bigg[Y=\pm \Big(\dd\frac{1-e^{-\beta}}{1+e^{-\beta}}\Big)\bigg]=1/2$.
\end{proposition}
The fact that on $G=\bb Z$ the Ant RW is a Markov chain is due to the specific structure of $\bb Z$. In general the process $(X_n)_{n\in\bb N}$ itself is \textit{not} a Markov chain. However, it is known that $\{\xi_n=(X_n,a_n),\, n\in \bb N\}$ does define one. 
We denote by $\bb P_{\xi}$ the law of this joint process when started from a given configuration $\xi_0=\xi$.

We introduce more notation.
Assume for the moment that $G$ is not a tree, so that in particular it possesses at least one circuit. Here, a circuit $C$ denotes a closed path of distinct \textit{directed} edges and distinct vertices. We will often write $C=(u_0,\ldots,u_{\ell-1})$ to denote a generic circuit $C$ of length $\ell$ with starting point (or root) $u_0$, where $u_i\neq u_j$ if $i\neq j$ and $u_{\ell-1}\sim u_0$.  
We  denote by $\mathscr{C}$ the set of all circuits on $G$.
 
For any $i\in \bb N=\{0,1,2,\ldots\}$ abbreviate $i(\ell)=i \mbox{ mod }\ell$.
We define the \textit{trapping event} associated to the circuit $C=(u_0,\ldots,u_{\ell-1})$ and the time $m\geq 0$ by
\begin{align}\label{eq2.5aa}
T^{C}_m\;=\;\big\{X_{m+i}\;=\;u_{i  (\ell)}\,, \forall\; i\geq 0\big\}\,.
\end{align}
In plain words, $T_m^{C}$ is the event in which the Ant RW is trapped in $C$  at time $m$, and afterwards spins around $C$ forever. We then define  
\begin{align}\label{trappingforever}
T^{C}\;=\; \bigcup_{m\geq 0} T^C_m\,,
\end{align}
which is the event that the Ant RW eventually gets trapped in $C$.
The main result of this paper is the following:
\begin{theorem}\label{thm2.1} Consider the  Ant RW $(X_n)_{n\in \bb N}$ with strength  of reinforcement $\beta\in(0,\infty)$ on an undirected graph   $G$ such that
\begin{enumerate}[a)] 
\item \label{item:finiteG} $G$ is connected, finite and is not a tree, or 
\item $G=\bb Z^d$ with $d\geq 2$.
\end{enumerate}
Then, 
\begin{equation*}
\bb P \Big(\bigcup_{C\in \mathscr{C}} T^C\Big)\;=\;1\,.
\end{equation*}
\end{theorem}
In words, under the above assumptions, the Ant RW will almost surely be eventually trapped  in some circuit $C$. Observe that, differently to random polymers or the Ising model, there is no phase transition in the parameter $\beta\in (0,\infty)$ and, differently to the usual symmetric random walk, the phase transition in the dimension occurs from $d=1$ to $d=2$.  To keep notation light, we simply assume that $\beta=1$ throughout the proofs, except in the proof of Proposition~\ref{th:it_is_markovian}. Going carefully over our proof it is however not hard to show that all results remain in force for any  $\beta\in (0,\infty)$.

Moreover, the proof of item b) of Theorem~\ref{thm2.1} can be easily adapted to different lattices. This is explained in Remark~\ref{lattice} where we point out which property a lattice must have in order to exhibit  the same behaviour as~$\bb Z^d$, $d\geq 2$, with respect to the Ant RW.

 \subsection{Idea of the proof for the finite case}
 \label{sec:idea-proof}
 
The central novelty  of this article is Theorem~\ref{thm2.1} for which we shortly explain the idea of its proof in the finite graph case.

To explain the idea of the proof we introduce our key concept, the \textit{good edge}. Let $v\in G$ and assume that $X_n=v.$ The probability of the walker to follow the edge $(v,w)$ is given by \eqref{eq:transitions} which can be rewritten as
\begin{equation}\label{eq:vw}
 \frac{1}{1+\!\!\!\displaystyle \sum_{\substack{u\,:\,u\sim v,\\ u\neq w}}\exp\big\{c_{n}(v,u)- c_{n}(v,w)\big\}}\,.
\end{equation} 
The main observation is that if $(v,w)$ is good in the sense that $c_n(v,w)$ maximises $c_n(v,u)$ over  $u\sim v,$ then \eqref{eq:vw} is bounded from below by $1/(1+D)$, where $D$ is the maximal degree of the graph $G$. In particular this bound is uniform in the environment. The major work in the proof of Theorem~\ref{trappingforever} is then to assure that the probability that the walker follows forever only good edges is bounded from below uniformly in the environment.

A renewal argument will then show that eventually this event will happen with probability one.
Since the graph is finite, a path consisting solely of good edges will eventually close a circuit, which can be shown to be followed forever.

\subsection{Open problems}
Theorem~\ref{thm2.1} gives a quite in depth description for the Ant RW on finite graphs. However, there are still many challenges left open, some of them which we plan to address in future works. We mention some of them:

$\bullet$   The weights in this article depend exponentially on the crossing numbers. It would be interesting to investigate the case in which the dependence is only of polynomial form. That is, for some $\gamma>0$, the environment $a_n$ is given via 
\begin{align*}
a_n(x,y)\;=\;\begin{cases} c_n(x,y)^\gamma, &\mbox{ if }\; c_n(x,y)>0\,,\\
1,&\mbox{ if }\; c_n(x,y)=0\,,\\
(-c_n(x,y))^{-\gamma}, &\mbox{ if }\; c_n(x,y)<0\,.
\end{cases}
\end{align*}
Does Theorem~\ref{thm2.1} still hold true? Is there maybe a phase transition in $\gamma$, in the sense that there exists $\gamma_*$ such that for $\gamma <\gamma_*$ the random walk does not necessarily get stuck in a circuit but for $\gamma>\gamma_*$ it does? If this is the case, does $\gamma_*$ depend on the choice of the graph, or is it maybe universal? We expect that to answer these questions the understanding of the role of the  environment on the dynamics employed in this article could be used. This however will not be enough. Indeed, one feature that is crucial to our analysis and to which the notion of good edge is well adjusted is that for any pair of edges $(x_1,y_1), (x_2,y_2)$ one has the relation
\begin{equation*}
\frac{a_n(x_1,y_1)}{a_n(x_2,y_2)}\;=\; \exp\big\{\beta[c_n(x_1,y_1)-c_n(x_2,y_2)]\big\}\,,
\end{equation*} 
which fails to be true in the polynomial case. In particular it is no longer enough to follow only good edges. In the polynomial case the lower bound will fail to be uniform and our arguments can not be applied directly.

$\bullet$  This work is mainly concerned with the Ant RW on finite graphs and on $\bb Z^d$. However, the behaviour of the walk on general infinite graphs can be very different, and can depend in a sensitive manner on the structure of the underlying graph. For instance, for the graph of Figure~\ref{fig1}, composed by   a circuit  connected to a copy of the infinite half line (we will call an infinite half line an \textit{infinite leaf}), the statement of Theorem~\ref{thm2.1} is not true.
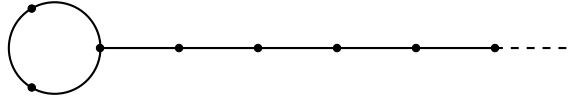
\begin{figure}[!htb]
\centering
\begin{tikzpicture}
\begin{scope}[scale=0.7]
\begin{scope}[scale=0.5]
\coordinate (a) at (0,0);
\coordinate (b) at (0,3);
\coordinate (c) at (1.5*1.73,1.5);
\coordinate (d) at (1.5*1.73+3,1.5);
\coordinate (e) at (1.5*1.73+6,1.5);
\coordinate (f) at (1.5*1.73+9,1.5);
\coordinate (g) at (1.5*1.73+12,1.5);
\coordinate (h) at (1.5*1.73+15,1.5);
\coordinate (i) at (1.5*1.73+17.9,1.5);
\coordinate (ccc) at (0.87,1.5);
\end{scope}

\draw[thick] (ccc) circle (0.87cm);
\draw[thick] (c)-- (d)--(e)--(f)--(g)--(h);
\draw[thick,dashed] (h)-- (i);
\filldraw[fill=black, draw=black] (a) circle (2pt)node[above]{};
\filldraw[fill=black, draw=black] (b) circle (2pt)node[right]{};
\filldraw[fill=black, draw=black] (c) circle (2pt)node[below]{};
\filldraw[fill=black, draw=black] (d) circle (2pt)node[below]{};
\filldraw[fill=black, draw=black] (e) circle (2pt)node[below]{};
\filldraw[fill=black, draw=black] (f) circle (2pt)node[below]{};
\filldraw[fill=black, draw=black] (g) circle (2pt)node[below]{};
\filldraw[fill=black, draw=black] (h) circle (2pt)node[below]{};

\end{scope}
\end{tikzpicture}
\caption{Graph $G$ given by a triangle connected to an infinite leaf.}
\label{fig1}
\end{figure}
In fact, using the same reasoning as in the proof of Proposition~\ref{th:it_is_markovian}, one can show that when walking over the infinite leaf the Ant RW behaves as an asymmetric random walk. In particular, it has positive probability of never returning to the root of the infinite leaf. Hence, the probability of not getting trapped in any circuit is  positive. More generally, any connected graph which is not a tree and possesses  an infinite leaf may serve as  example as well.
The presence of an infinite leaf is sufficient to assure  that, with positive probability, the Ant RW is not trapped in any  circuit, but we believe that it should be not necessary. Hand-waving calculations guided us to guess that ``\textit{an infinite tree whose nodes at even generations are replaced by circuits with a sufficiently large number of branches leaving from it}''  should be such a corresponding example  (see Figure~\ref{fig1b} for an illustration).
\begin{figure}[htb!]
\centering
\begin{tikzpicture}
\begin{scope}[scale=0.7]

\coordinate (a) at (0,0);
\coordinate (b) at (-0.5,0);
\coordinate (c) at (0.5,0);
\coordinate (e) at (2.5,-2);

\coordinate (d) at (-2.5,-2);
\coordinate (f) at (-3.5,-3);
\coordinate (g) at (-1.5,-3);
\coordinate (h) at (-3.5,-3.5);
\coordinate (h1) at (-4,-3.5);
\coordinate (h2) at (-3,-3.5);
\coordinate (h3) at (-3.5,-4);
\coordinate (g0) at (-1.5,-3.5);
\coordinate (g1) at (-2,-3.5);
\coordinate (g2) at (-1,-3.5);
\coordinate (g3) at (-1.5,-4);

\coordinate (h4) at (-3.5,-6);
\coordinate (A) at (-6.5,-7);
\coordinate (B) at (-4.5,-7);
\coordinate (C) at (-2.5,-7);
\coordinate (D) at (-0.5,-7);

\centerarc[thick,dashed](-6.5,-7.5)(10:170:0.5);
\centerarc[thick,dashed](-4.5,-7.5)(10:170:0.5);
\centerarc[thick,dashed](-2.5,-7.5)(10:170:0.5);

\begin{scope}[xshift=4.5cm]
\coordinate (A0) at (-5,-7.5);
\coordinate (A1) at (-5.5,-7.5);
\coordinate (A2) at (-4.5,-7.5);
\coordinate (A3) at (-5,-8);
\coordinate (B1) at (-5+0.35,-7.5+0.35);
\coordinate (B2) at (-5-0.35,-7.5+0.35);
\coordinate (B3) at (-5+0.35,-7.5-0.35);
\coordinate (B4) at (-5-0.35,-7.5-0.35);
\end{scope}

\coordinate (Bd) at (2.5,-2);
\coordinate (Bf) at (1.5,-3);
\coordinate (Bg) at (3.5,-3);
\coordinate (Bh) at (1.5,-3.5);
\coordinate (Bh1) at (1,-3.5);
\coordinate (Bh2) at (2,-3.5);
\coordinate (Bh3) at (1.5,-4);
\coordinate (Bg0) at (3.5,-3.5);
\coordinate (Bg1) at (3,-3.5);
\coordinate (Bg2) at (4,-3.5);
\coordinate (Bg3) at (3.5,-4);

\draw[thick] (a) circle (0.5cm);
\draw[thick] (b)--(d);
\draw[thick] (c)--(e);
\draw[thick] (d)--(f);
\draw[thick] (d)--(g);
\draw[thick] (h3)--(h4);
\draw[thick] (h) circle (0.5cm);
\draw[thick] (g0) circle (0.5cm);

\draw[thick] (h4)--(A);
\draw[thick] (h4)--(B);
\draw[thick] (h4)--(C);
\draw[thick] (h4)--(D);

\draw[thick] (A0) circle (0.5cm);

\filldraw[fill=black, draw=black] (b) circle (2pt)node[above]{};
\filldraw[fill=black, draw=black] (c) circle (2pt)node[above]{};
\filldraw[fill=black, draw=black] (d) circle (2pt)node[above]{};
\filldraw[fill=black, draw=black] (f) circle (2pt)node[above]{};
\filldraw[fill=black, draw=black] (g) circle (2pt)node[above]{};

\filldraw[fill=black, draw=black] (h1) circle (2pt)node[above]{};
\filldraw[fill=black, draw=black] (h2) circle (2pt)node[above]{};
\filldraw[fill=black, draw=black] (h3) circle (2pt)node[above]{};
\filldraw[fill=black, draw=black] (h4) circle (2pt)node[above]{};

\filldraw[fill=black, draw=black] (g1) circle (2pt)node[above]{};
\filldraw[fill=black, draw=black] (g2) circle (2pt)node[above]{};
\filldraw[fill=black, draw=black] (g3) circle (2pt)node[above]{};

\filldraw[fill=black, draw=black] (A) circle (2pt)node[above]{};
\filldraw[fill=black, draw=black] (B) circle (2pt)node[above]{};
\filldraw[fill=black, draw=black] (C) circle (2pt)node[above]{};
\filldraw[fill=black, draw=black] (D) circle (2pt)node[above]{};

\filldraw[fill=black, draw=black] (A1) circle (2pt)node[above]{};
\filldraw[fill=black, draw=black] (A2) circle (2pt)node[above]{};
\filldraw[fill=black, draw=black] (A3) circle (2pt)node[above]{};

\filldraw[fill=black, draw=black] (B1) circle (2pt)node[above]{};
\filldraw[fill=black, draw=black] (B2) circle (2pt)node[above]{};
\filldraw[fill=black, draw=black] (B3) circle (2pt)node[above]{};
\filldraw[fill=black, draw=black] (B4) circle (2pt)node[above]{};

\draw[thick] (c)--(Bd);
\draw[thick] (Bd)--(Bf);
\draw[thick] (Bd)--(Bg);
\draw[thick] (Bh) circle (0.5cm);
\draw[thick] (Bg0) circle (0.5cm);

\filldraw[fill=black, draw=black] (Bd) circle (2pt)node[above]{};
\filldraw[fill=black, draw=black] (Bf) circle (2pt)node[above]{};
\filldraw[fill=black, draw=black] (Bg) circle (2pt)node[above]{};

\filldraw[fill=black, draw=black] (Bh1) circle (2pt)node[above]{};
\filldraw[fill=black, draw=black] (Bh2) circle (2pt)node[above]{};
\filldraw[fill=black, draw=black] (Bh3) circle (2pt)node[above]{};

\filldraw[fill=black, draw=black] (Bg1) circle (2pt)node[above]{};
\filldraw[fill=black, draw=black] (Bg2) circle (2pt)node[above]{};
\filldraw[fill=black, draw=black] (Bg3) circle (2pt)node[above]{};

\draw[dashed] (h2)--(-3,-4.5);
\draw[dashed] (h1)--(-4,-4.5);

\draw[dashed] (g3)--(-1.5,-5);
\draw[dashed] (g2)--(-1,-4.5);
\draw[dashed] (g1)--(-2,-4.5);

\draw[dashed] (Bh3)--(1.5,-5);
\draw[dashed] (Bh1)--(1,-4.5);
\draw[dashed] (Bh2)--(2,-4.5);

\draw[dashed] (Bg3)--(3.5,-5);
\draw[dashed] (Bg1)--(3,-4.5);
\draw[dashed] (Bg2)--(4,-4.5);

\end{scope}
\end{tikzpicture}
\caption{Infinite graph $G$ which is not a tree, has no infinite leaf, and for which we believe  the Ant RW has  positive probability of not getting trapped in any circuit. Since the number of branches as well as  circuit lengths increase exponentially  as we step forward to next generations, we believe that the probability of never going back to a previous generations and also never closing a single circuit is positive.}
\label{fig1b}
\end{figure}
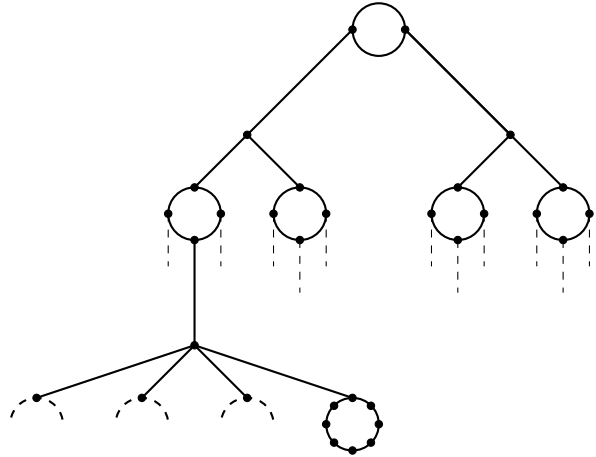
In light of the above discussion and  Theorem~\ref{thm2.1}, we also conjecture:
 \begin{conjecture}
 Let $G$ be an infinite graph. Then, denoting by $d(\cdot,\cdot)$ the shortest path distance on $G$, one has the following dichotomy
 \begin{equation*}
\bb P \Big(\bigcup_{C\in \mathscr{C}} T^C \cup\big\{\lim_{n\to\infty}d(X_n, X_0)= \infty\big\}\Big)\;=\;1\,,
 \end{equation*}
 i.e., either the walk gets trapped in a circuit or it escapes to infinity.
 \end{conjecture}
 
$\bullet$ The Ant Mill phenomenon alluded to above is observed in a \textit{group} of army ants. Thus, it would actually be more natural to study the behaviour of a large number of Ant RWs.  In this case there will be two effects that are competing with each other. On the one hand if an ant follows an edge already crossed before by another ant it further reinforces the edge, so that it should be easier for a large group of ants to be trapped in a circuit. However, as long as the reinforcement is not yet strong enough an ant may also simply cross a directed edge in the opposite direction and in this way kill the reinforcement effect and ``neutralize'' the edge. It would be interesting to investigate the localisation behaviour in the case of a group of Ant RWs.

\section{Proof of Theorem~\ref{thm2.1} in the finite graph case}\label{sec:3}

To prove Theorem~\ref{thm2.1} in the finite graph case we will show that eventually the walk only follows good edges, i.e., edges that maximise their crossing numbers among adjacent edges. By the explanation given in Section~\ref{sec:idea-proof} it is tempting to impose that the path only follows good edges. However, it may be possible that  the path ends up in a leaf and then is stuck forever on the edge adjacent to that leaf. In other words one must guarantee that the walker never backtracks, i.e., never goes back to a vertex visited immediately before. To that end a certain structure on the good edges needs to be required. We now formalise these ideas.

\begin{definition} Let $(X_n)_{n\in\bb N}$ be  the Ant Random Walk on $G$. Given a vertex $u\in G$ we say that the edge $(u,v)\in E$ is a good edge for $u$ at time $n$ if
\[c_n(u,v)\;=\;\max_{w\,:\,w\sim u}c_n(u,w)\,.\]
\end{definition}

\ignore{
\begin{remark}\rm
The above definition of course implies that to each vertex $u\in V$ there exists at least one corresponding good edge. Note that
if $(u,v)$ is a good edge for $u$ then
\[\min_{\substack{w\,:\,w\sim u,\\ w\neq v}}\big[c_n(u,v)-c_n(u,w)\big]\;\geq\; 0\,,\]
which in view of~\eqref{eq:RnC} will help to construct a \textit{good cycle}, that is, a cycle composed of good edges.
\end{remark}
}

Since $G$ is connected for any non-empty subset $S\subset G$ and any vertex $v\in G$ there is a path connecting $v$ to any vertex of $S$. We denote by $v\to S$ a shortest such path connecting $v$ to $S$. 
With a slight abuse of notation we also denote by $\{X_n \to S\}$ the event that the Ant RW walks along a shortest path connecting the random vertex $X_n$ with $S$.
We further fix a circuit $C^*$ of $G$.
For any finite stopping time $\tau$ we define the following random set
\begin{equation}\label{def:Stau}
S_\tau\;=\;\big\{u\in V:\exists\; v\sim u \mbox{ such that }|c_\tau(u,v)|\geq 2\big\}\,.
\end{equation}

Now we define an auxiliary random path $(Y^\tau_n)_{n\geq 0}$ that is a deterministic function of the pair $(S_\tau,X_\tau).$ We start with  $Y^\tau_0=X_\tau$ and $(c^\tau_0(\cdot,\cdot))=(c_\tau(\cdot,\cdot)).$
The evolution of the field $(c^\tau_n(\cdot,\cdot))_n$ will obey the same rules as in \eqref{eq:edge_weights} with $Y$ instead of $X$. We distinguish the following cases.

\begin{itemize}
	\item[1)] $S_\tau\neq \varnothing,$ and $X_\tau \in S_\tau$. The construction of $(Y^\tau_n)_n$ goes as follows. Assume that for some $j\geq 1$,  $Y_0^\tau, Y_1^\tau,\ldots, Y_{j-1}^\tau$ have already been constructed. We then choose $Y_j^\tau$ such that the edge $(Y_{j-1}^\tau, Y_j^\tau)$ is good. We remark   that we will show in Section~\ref{sec:nonback} that $Y_j^\tau\neq Y_{j-2}^\tau$. Moreover, since $G$ is finite $(Y_n^\tau)_{n}$ eventually will visit a vertex for the second time and thereafter will follow forever a circuit composed of good edges. 
	\item[2)] $S_\tau\neq \varnothing$, and $X_\tau \notin S_\tau$.  In that case $(Y^\tau_n)_n$ first follows the path $X_\tau\to S_\tau$. Having reached $S_\tau$ it copies the strategy from the first item, and therefore will eventually follow a circuit of good edges forever.
	\item[3)] $S_\tau=\varnothing$, and $X_\tau \notin C^*$. In that case $(Y^\tau_n)_n$ follows $X_\tau\to C^*$ and after that gives infinite turns around $C^*.$ 
	\item[4)] $S_\tau=\varnothing$, and $X_\tau\in C^*$. Then $(Y^\tau_n)_n$ just gives infinite turns around $C^*.$
\end{itemize}
Note that in all four cases above, the path $(Y_n^\tau)_n$ will eventually follow a circuit of good edges.

Given the above construction we then define recursively the following sequence of stopping times: $\tau_0=0$, and 
\[\tau_{k+1}\;=\;\inf\{n>\tau_k \,:\, X_n\neq Y^{\tau_k}_{n-\tau_k}\}\,.\]
Where we use the convention that $\tau_{k+1}=\infty$ if $\tau_k=\infty$.

The crucial observation is that on the event $\{\exists\, k\geq 1\,:\,\tau_k=\infty\}$ the walker $(X_n)_{n\geq 1}$ will eventually be trapped in a circuit.

\begin{lemma}\label{lemma:uniform-bound} There exists a constant $\delta=\delta(G)>0$ such that almost surely for any $k\in\bb N$ 
\[\bb P_{\xi_{\tau_k}}(\tau_{k+1}=\infty )\;\geq\; \delta\quad \mbox{ on the event $\{\tau_k<\infty\}$}\,.\]
\end{lemma}

We prove Lemma~\ref{lemma:uniform-bound} in Section \ref{sec:uniform-bound}. We now show how to deduce the first item in Theorem~\ref{thm2.1} from Lemma~\ref{lemma:uniform-bound}. 

Our goal is to prove that for an initial state $\xi_0=(X_0,c_0(\cdot,\cdot)),$ with $c_0(\cdot,\cdot)\equiv 0,$ one has that $\bb P_{\xi_0}(\exists\, k\geq 1\,:\,\tau_k=\infty)=1.$ By the Borel-Cantelli Lemma it is enough to show that
\begin{equation}\label{eq:finite-mean}
\sum_{k\geq 1}\bb P_{\xi_0}(\tau_k<\infty)\;<\;\infty\,.
\end{equation}
Note that by Lemma~\ref{lemma:uniform-bound} applied to $k=0$,
\[\bb P_{\xi_0}(\tau_1<\infty)\;=\;1-\bb P_{\xi_0}(\tau_1=\infty)\;\leq\; 1-\delta\,.\]
 Assume that \[\bb P_{\xi_0}(\tau_k<\infty)\leq (1-\delta)^k.\] Using the Markov property and again Lemma \ref{lemma:uniform-bound} it follows that
\begin{align*}
\bb P_{\xi_0}(\tau_{k+1}<\infty)&\;=\;\bb P_{\xi_0}(\tau_{k+1}<\infty,\,\tau_k<\infty)\\
&\;=\;\bb E_{\xi_0}\big[\one \{\tau_k<\infty\} \bb P_{\xi_{\tau_k}}(\tau_{k+1}<\infty)\big]\\
&\;\leq\; (1-\delta)^{k+1}
\end{align*}
and this concludes the proof of \eqref{eq:finite-mean}.

\subsection{Non backtracking property}
\label{sec:nonback}
The goal of this section is to prove that with certain control on the environment the strategy of  following only good edges generates a non-backtracking path. To be more precise, we want to show that the path $(Y_n^\tau)_n$ constructed in the first item in the previous section does not make steps back to a vertex visited one time unit before. It will be expedient to study some flow properties of the crossing numbers.
\begin{definition} The total flow at time $n$ through the vertex $u\in G$ is defined by the quantity
\[F_n(u)\;=\;\sum_{v\,:\,v\sim u}c_n(u,v)\,.\]
\end{definition}
The next result is a simple fact about the flow of the random walk on a graph, which is well known in the field, and so we omit its proof.
\begin{lemma}\label{lema:flow} In the previous setting, fix a vertex $u\in G.$ Then $F_n(u)\in$\break$\{-1,0, +1\}$.
Furthermore, either $F_n(u)=0$ for any $u\in G$ or there exist exactly two vertices $v_{i}$ and $v_f$, called the initial and the final vertex, respectively, such that $F_n(v_{i})=+1$ and  $F_n(v_{f})=-1$ and $F_n(u)=0$ for the remaining vertices. 
\end{lemma}
A consequence of the previous result is the following proposition.
\begin{proposition}\label{prop:good>0} If $X_0^n=\{X_0,\ldots, X_n\}$ is a closed path, i.e., $X_0=X_n$, and $c_n(u,w)\neq 0$ for some vertex $w\sim u$ then, for any good edge $(u,v)$ of $u$, one has that $c_n(u,v)>0.$
\end{proposition}
\begin{proof}
We first note that since $X_0=X_n$, Lemma~\ref{lema:flow} implies that $F_n(u)= 0$ for all $u\in V$.  If $c_n(u,w)>0$ then the good edge has a positive crossing number since it maximises the crossing numbers among the neighbours of $u$. If $c_n(u,w)<0$, then we can conclude from $F_n(u)=0$, that there exists a vertex $w^*$ such that $c_n(u,w^*)>0$.  Hence, the claim follows.
\end{proof}
Of course there is no guarantee that the Ant RW at time $n$ does form a closed path. However, in any case we have the following result.
\begin{proposition}\label{prop:good>1} Let $u$ be a given vertex, and fix a realisation $X_0^n=\{X_0,\ldots,$\break$ X_n\}$ of the Ant RW until time $n$. Assuming that there exists an edge $(u,w)$ such that $c_n(u,w)\leq -2$, then any good edge $(u,v)$ of $u$ satisfies $c_n(u,v)\geq 1.$
\end{proposition}
\begin{proof}
Since $c_n(u,w)\leq -2$ and $F_n(u)=\sum_{z:z\sim u}c_n(u,z)\geq -1$ (cf. Lemma~\ref{lema:flow}) it is impossible that $c_n(u,w)\leq 0$ for all $w\sim u$. Therefore, any good edge $(u,v)$ of $u$ satisfies $c_n(u,v)\geq 1$.
\end{proof}

Now we will show that if
$S_\tau\neq \varnothing$, and $X_\tau=v_0 \in S_\tau$, then $(Y^\tau_n)_n$ is non backtracking. Recall that $Y^\tau_0=X_\tau=v_0.$ By Proposition \ref{prop:good>1}, any neighbour $v_1$ of $Y^\tau_0$ such that $(Y^\tau_0,v_1)$ is a good edge satisfies $c^\tau_0(Y^\tau_0,v_1)\geq 1$. To proceed assume $Y^\tau_1=v_1$. As consequence, $c^\tau_{1}(Y^\tau_0,Y^\tau_{1})\geq 2,$ which implies $c^\tau_{1}(Y^\tau_{1},Y^\tau_0)=-c^\tau_{1}(Y^\tau_0,Y^\tau_1)\leq -2.$ Thus, we can again apply Proposition~\ref{prop:good>1}. Consequently, for all neighbours $v_2$ of $Y^\tau_{1}$ such that $(Y^\tau_{1},v_2)$ is a good edge one has that $c_{1}^\tau(Y^\tau_{1},v_2)\geq 1$. Note that $v_2\neq Y^\tau_0$ since $c_{1}^\tau(Y^\tau_{1},Y_0^\tau)\leq -2.$ In particular, $Y^\tau_{1}=v_1$ has degree at least $2.$ The key observation is that for any $v_2$ as above the path $\{Y^\tau_{0}=v_0,\;Y^\tau_{1}=v_1,\; Y^\tau_{2}=v_2\}$ is non-backtracking and that $c^\tau_{2}(Y^\tau_{1},Y^\tau_{2})\geq 2$. 

Repeatedly applying the above arguments shows that $(Y_n^\tau)$ never backtracks. Moreover, it also shows that since the graph $G$ is finite, and $(Y_n^\tau)$ walks along vertices of degree at least $2$,  it eventually follows a circuit consisting of strictly positive good edges.

\subsection{Proof of Lemma \ref{lemma:uniform-bound}}
\label{sec:uniform-bound}
 We want to prove that $\bb P_{\xi_{\tau_k}}(\tau_{k+1}=\infty)\geq \delta$ on the set where $\tau_k$ is finite. This is implied by
\[\bb P_{\xi_{\tau_k}}\big(\forall\, n\geq \tau_k,\,X_n=Y^{\tau_k}_{n-\tau_k}\big)\;\geq\; \delta\,.\]

To prove the above statement we will need to estimate the probability of giving turns around a circuit. If $C=(u_0,\ldots, u_{\ell-1})$ is a circuit and $X_n=u_0=u_\ell$, then the probability of making a turn around $C$ is given by
\begin{equation}\label{eq:TC1}
 \prod_{j=0}^{\ell-1}\frac{1}{1+\!\!\!\displaystyle \sum_{\substack{w\,:\,w\sim u_{j},\\ w\neq u_{j+1}}}\exp\big\{c_{n+j}(u_{j},w)- c_{n+j}(u_{j},u_{j+1})\big\}}\,.
\end{equation}

To analyse~\eqref{eq:TC1} it comes in handy to introduce the quantity $R_n^C$ defined via
\begin{equation}\label{eq:RnC}
R_n^C\;=\; \min_{0\leq j\leq \ell-1}\min_{\substack{y\,:\,y\sim u_j,\\ y\neq u_{j+1}}}
\big[c_n(u_j, u_{j+1})- c_n(u_j, y)\big].
\end{equation}
Recall the definition of the event $T^C_m$ in \eqref{eq2.5aa}.
\begin{proposition}\label{lemma:trapped} Assume that $R_n^C\geq -2$ and let $|V|$ be the number of vertices of $G$ and $D$ be the maximum degree of $G$. Then for any configuration $\xi_0$
\[\bb P_{\xi_0}(T^C)\geq \exp\bigg(\frac{-|V|De^2}{1-e^{-1}}\bigg).\]
\end{proposition}
We prove this proposition in Section~\ref{sec:trapped}.
Now we proceed to prove Lemma~\ref{lemma:uniform-bound}.
We distinguish between several cases.

\textbf{(1)}
 $S_{\tau_k}=\varnothing$. In this case, the absolute value of all crossing numbers are bounded by one.
  Observe that on the event $\{S_{\tau_k}=\varnothing\}$ one has that $R_{\tau_k}^{C^*}\geq -2$. Define the stopping time $\sigma=\inf\{m\geq 0\,:\,Y^{\tau_k}_m \in C^*\}$. We have the equality \[\{\forall\, n\geq \tau_k,\,X_n=Y^{\tau_k}_{n-\tau_k}\}=\{X_{\tau_k}\to C^*\}\cap T^C_\sigma.\]
  
Therefore, we need to estimate the probability
\begin{equation*}
\bb P_{\xi_{\tau_k}} \bigg[\{X_0 \to C^*\}\cap T_{\sigma}^{C^*}\bigg]\,.
\end{equation*}
 
In the following we will use the Markov Property at time $\sigma.$ Notice that on the event $\{X_0\to C^*\}$ one still has that $R_{\sigma}^{C^*}\geq -2$ and it always holds that $|C|\leq |V|$. Then with the help of Proposition~\ref{lemma:trapped},  we can estimate, almost surely, the above probability from below by
\begin{equation*}
\bb E_{\xi_{\tau_k}}\Big[\one_{\{X_0 \to C^*\}}\bb P_{\xi_{\sigma}}\big[T^{C}\big]\Big]
\;\geq\; \bb P_{\xi_0}\big[X_{\tau_k}\to C^*\big] \exp\bigg(\frac{-|V|De^2}{1-e^{-1}}\bigg)\,.
\end{equation*}
It only remains to bound the probability on the right hand side above. To that end note that along the path $X_{\tau_k}\to C^*$ all edges have crossing number bounded in modulus by one and that $\sigma\leq |V|.$ Hence, for any fixed vertex $v\in V,$ on the event $\{X_{\tau_k}=v, S_{\tau_k}=\varnothing\}$,
\begin{equation*}
\bb P_{\xi_{\tau_k}}\big[X_0\to C^*\big]\;=\;\bb P_{v}\big[X_0\to C^*\big] \;\geq\; \frac{1}{(1+D e^{2})^{|V|}}\,.
\end{equation*}
Thus, we can conclude that on $\{S_{\tau_k}=\varnothing\}$
\begin{equation*}
\bb P_{\xi_{\tau_k}}\big(\forall\, n\geq \tau_k,\,X_n=Y^{\tau_k}_{n-\tau_k}\big)
\geq\; \frac{1}{(1+D e^{2})^{|V|}}\,\exp\bigg(\frac{-|V|De^2}{1-e^{-1}}\bigg)\;:= \;\delta^{(1)}\,.
\end{equation*}

\textbf{(2)}
$S_{\tau_k}\neq \varnothing$, and $X_{\tau_k}=v_0 \in S_{\tau_k}$. By definition on the event
\[\{\forall\, n\geq \tau_k,\,X_n=Y^{\tau_k}_{n-\tau_k}\}\] 
the walker will follow only good edges and eventually closes a good circuit $C$ and gives infinite turns around it. Let $\sigma$ be the first time the walker meets $C$. Since $C$ is a good circuit, then $R_\sigma^C\geq 0> -2$. Again the path joining $X_{\tau_k}$ with $C$ is of length at most $|V|$. Hence, similarly as in the case $S_{\tau_k}=\varnothing$, we see that on $\{S_{\tau_k}\neq\varnothing\}\cap\{X_{\tau_k}\notin C\}$
\begin{equation*}
\bb P_{\xi_{\tau_k}}\big(\forall\, n\geq \tau_k,\,X_n=Y^{\tau_k}_{n-\tau_k}\big)
\geq\; \frac{1}{(1+D)^{|V|}}\,\exp\bigg(\frac{-|V|De^2}{1-e^{-1}}\bigg)\;:= \;\delta^{(2)}\,.
\end{equation*}
where the terms $e^2$ are not present since the path connecting $X_{\tau_k}$ with $C$ consists only of good edges.\smallskip

\textbf{(3)} $S_{\tau_k}\neq \varnothing$, and $X_{\tau_k} \notin S_{\tau_k}.$ Define  $\sigma=\inf\{m\geq 0 \,:\,Y^{\tau_k}_m \in S_{\tau_k} \}.$ Denote by $v_0$ a random vertex in $S_{\tau_k}$ that minimizes the graph distance to $X_{\tau_k}=Y^{\tau_k}_0$ among all vertices in $S_{\tau_k}$.  In this case the path $Y^{\tau_k}_{0}\to v_0$ lies completely in $S_{\tau_k}^\complement,$ i.e., all edges on this path have crossing numbers bounded in absolute value by one. 
Hence, as in the case \textbf{(1)} we obtain, almost surely,
\begin{equation*}
\bb P_{\xi_{\tau_k}}\big[X_0\to S_0\big]\;\geq\; \frac{1}{(1+De^2)^{|V|}}\,.
\end{equation*}

On the event $\{Y^{\tau_k}_0\to S_{\tau_k}\}$ we then have that $Y^{\tau_k}_{\sigma}\in S_{\tau_k}\subset  S_{\tau_k+\sigma}$. In particular, on the event $\{\forall\,n\geq \tau_k,\,X_n=Y^{\tau_k}_{n-\tau_k}\}$ we have that $X_{\tau_k+\sigma}\in S_{\tau_k+\sigma}$ and we are in the setting of the previous case. Therefore, on the event $\{S_{\tau_k}\neq \varnothing\}\cap\{X_{\tau_k}\notin S_{\tau_k}\}$ we can bound
\begin{align*}
&\bb P_{\xi_{\tau_k}}\big(\forall\, n\geq \tau_k,\,X_n=Y^{\tau_k}_{n-\tau_k}\big)\\
&\geq\; \frac{1}{(1+De^2)^{|V|}}\,\frac{1}{(1+D)^{|V|}}\exp\bigg(\frac{-|V|De^2}{1-e^{-1}}\bigg)\;:=\; \delta^{(3)}\,.
\end{align*}
Collecting all estimates obtained we see that the proof is concluded with the choice $\delta=\min\big\{\delta^{(1)}, \delta^{(2)}, \delta^{(3)}\big\}=\delta^{(3)}$.

\subsection{Proof of Proposition~\ref{lemma:trapped}}
\label{sec:trapped}
Fix a circuit $C=(u_0,\ldots,u_{\ell-1})$ and recall our notation $i(\ell)=i \mbox{ mod }\ell$. For $k\in \bb N$  we define the \textit{truncated trapping event} $T_m^{C,k}$ by
\begin{equation}\label{eq:truncated}
T_m^{C,k} \;= \; \{X_{m+i}\;=\;u_{i  (\ell)}, \forall\; 0\leq i\leq k\ell\}\,.
\end{equation}
In plain words, $T_m^{C,k}$ is the event in which the walk makes  $k$ consecutive turns around $C$ starting at time $m$.

To continue bounding the above trapping event we adopt a notation in this section that slightly differs from the one used in the rest of the article. For a field of integers $\{c_0(x,y):\, (x,y)\in E\}$ we denote with a slight abuse of notation
\[c_n (x,y) \;=\;c_0(x,y)+ \sum_{k=0}^{n-1}\Big( \one\big[(X_k,X_{k+1})=(x,y)\big]-\one\big[(X_k,X_{k+1})=(y,x)\big]\Big)\,.\]
Recall that we write $\xi_n=(X_n,a_n)$ for the pair consisting of the position of the Ant RW and its induced environment at time $n$. Let $\xi_0=(u_0,a_0)$ be  the initial state of this Markov chain. We then have the following:
\begin{lemma}\label{le:1st-turn}
Let $G$ be any locally finite graph. If $X_0=u_0$, then 
\begin{equation*}
\bb P_{\xi_{0}}(T_0^{C,1})
\;\geq\;  \prod_{j=0}^{\ell-1}\frac{1}{1+\displaystyle \sum_{\substack{w\,:\,w\sim u_{j}\\ w\neq u_{j+1}}}\exp\big\{c_{0}(u_{j},w)- c_{0}(u_{j},u_{j+1})\big\}}\,.
\end{equation*}
\end{lemma}

\begin{proof}
Observe that on the event $T_0^{C,1}$ the walker makes one turn around $C$. Therefore, on that event we have, for all $0\leq j\leq \ell-1$,
\begin{align*}
\begin{cases}
c_0(u,v)=c_{j}(u,v),&\mbox{ if }(u,v)\notin C,\\
c_{j}(u,v)\geq c_{0}(u,v),&\mbox{ if } (u,v)\in C.
\end{cases}
\end{align*}
 It is now plain to see that for all $j\in\{0,\ldots, \ell-1\}$ and for all $w\sim u_j,$ $w\neq u_{j}$,
\begin{equation*}
c_{j}(u_{j},w)- c_{j}(u_{j},u_{j+1}) \;\leq\; c_{0}(u_{j},w)- c_{0}(u_{j},u_{j+1})\,.
\end{equation*}
Hence, the claim follows from equation~\eqref{eq:TC1}.
\end{proof}

\begin{lemma}\label{le:kturns} Let $G$ be a locally finite graph with maximum degree $D<\infty.$ For all $k\in \bb N$ and  all $M\in \bb R,$ if $R_0^C\geq M$ and $X_n=u_0$, then
\begin{equation*}
\bb P_{\xi_0}(T_0^{C,k})
\;\geq\; \frac{1}{\prod_{j=0}^{k-1}\big(1+ D\exp\{-M-j\}\big)^\ell}\,.
\end{equation*}
\end{lemma}
\begin{proof}
We prove the result by induction. The case $k=1$ is an immediate consequence of Lemma \ref{le:1st-turn} using that $D$ is a bound for the degree of any vertex of~$G$. 

Assume that the result is true for $1, \ldots, k-1$ and for all $M\in \bb R.$ Using the Markov Property and the observation $T_{0}^{C,k} = T_{0}^{C,k-1}\cap T_{(k-1)\ell}^{C,1}$,  we obtain that
\begin{align*}
\bb P_{\xi_0}\big[T_{0}^{C,k}\big]&\;=\;\bb E_{\xi_0}\big[\bb P_{\xi_0}(T_{0}^{C,k} \,\vert\, \mc G_{(k-1)\ell})\big]\\
&\;=\;\bb E_{\xi_0}\big[\bb P_{\xi_0}\big(T_{0}^{C,k-1}\cap T_{(k-1)\ell}^{C,1} \,\vert\, \mc G_{(k-1)\ell}\big)\big]\\
&\;=\;\bb E_{\xi_0}\big[\one_{T_{0}^{C,k-1}}\bb P_{\xi_{(k-1)\ell}}\big(T_{0}^{C,1}\big)\big]\,.
\end{align*}
Now observe that on the event $T_{0}^{C,k-1}$, we have that $R^C_{(k-1)\ell}\geq M+k-1.$ Therefore, using the base case $k=1$, we can write
\begin{align*}
\bb E_{\xi_0}\big[\one_{T_0^{C,k-1}}\bb P_{\xi_{(k-1)\ell}}(T_0^{C,1})\big] &\;\geq\; \bb E_{\xi_0}\Big[\one_{T_0^{C,k-1}}\frac{1}{(1+ D\exp\{-M-(k-1)\})^\ell}\Big]\\
&\;=\; \frac{1}{(1+ D\exp\{-M-(k-1)\})^\ell}\,\bb P_{\xi_0}\big(T_0^{C,k-1}\big)
\end{align*}and using the induction hypothesis we  finish the proof.
\end{proof}

To finish the proof of Proposition~\ref{lemma:trapped} we just use the following facts:
\begin{enumerate}
\item $\bb P_{\xi_0}\big(T_0^C\big)=\lim_{k\to \infty} \bb P_{\xi_0}\big(T_0^{C,k}\big)$,
\item $1+x\leq e^x,$
\item $\sum_{i\geq 0}x^i=1/(1-x)$ if $x\in [0,1).$
\end{enumerate}

\section{Proof  of \texorpdfstring{Theorem~\ref{thm2.1}}{Theorem 2.1} in the \texorpdfstring{$\bb Z^d$}{Zd} case} \label{sec:Zd}

We start now to deal with the proof of Theorem~\ref{thm2.1} in the  case $G=\bb Z^d$ with $d\geq 2$. A sketch of the proof goes as follows. First, we  will argue that the Ant RW has a uniformly bounded from below probability of being trapped in a circuit right after escaping certain increasing balls. This will show that the walk is almost surely bounded. Then, we will construct a simultaneous coupling between the Ant RW on $\bb Z^d$ and on all those balls. Under that coupling and by the previous boundedness result, we will conclude that the Ant RW on $\bb Z^d$ almost surely coincides with the Ant RW on some (random) ball. This with the finite case of Theorem~\ref{thm2.1} will permit to conclude the proof.
\begin{proposition}\label{prop62}
The Ant RW in $G=\bb Z^d$ with $d\geq 2$ is almost surely bounded.
\end{proposition}
\begin{proof}
Denote by $B_k=B[0,k]$ the closed ball of center $0$ and radius $k\in \bb N$ in the graph $\bb Z^d$ with respect to the $\ell^1$-distance and denote by $\p B_{k}$ its inner boundary.  Recall that we are assuming $X_0=0$. For each $k\in\bb N$ we define the stopping time
\begin{equation}\label{stopping}
\tau_k\;=\;\inf\{n>0: X_n \in  B_{3k}^\complement\}
\end{equation}
and let
\begin{align*}
E_k\;\overset{\text{def}}{=}\; \big\{\tau_k<\infty\big\}\,.
\end{align*}
That is, $E_k$ is the event where the Ant RW escapes the ball of radius $3k$. Let $V(k)$ be the set of vertices $v\in \bb Z^d$ such that $d(v,B_{3k})=1$. It is elementary to check that, for each $v \in V(k)$, there exists a circuit $C_{v}$ of length $4$ such that $C_{v}\subset B_{3(k+1)}\backslash B_{3k}$, see Figure~\ref{Fk} for an illustration. These circuits are not unique. However in the sequel, for ease of notation, for each $v$ as above $C_{v}\subset B_{3(k+1)}\backslash B_{3k}$ denotes a fixed but arbitrarily chosen circuit.

 Recall the trapping event $T^C_m$ defined in \eqref{trappingforever} and let
\begin{align*}
F_k\;\overset{\text{def}}{=}\; E_k\bigcap \Big(\bigcup_{v\in V(k)} T^{C_{v}}_{\tau_k} \Big)\,.
\end{align*}
In other words,  $F_k$ is the event in which the Ant RW eventually escapes $B_{3k}$  and immediately after that, is trapped in a directed circuit of length four contained in $B_{3(k+1)}\backslash B_{3k}$, see  Figure~\ref{Fk}.
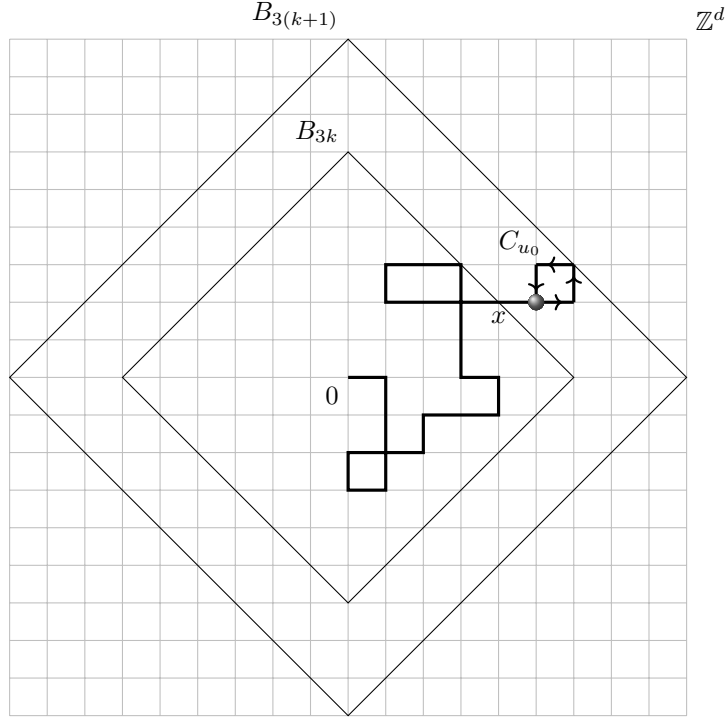
\begin{figure}[!htb]
\centering
\begin{tikzpicture}
\draw[step=0.5,lightgray] (-4.5,-4.5) grid (4.5,4.5);
\draw (0,3)--(3,0)--(0,-3)--(-3,0)--cycle;
\draw (0,4.5)--(4.5,0)--(0,-4.5)--(-4.5,0)--cycle;

\draw [very thick] (0,0)--(0.5,0)--(0.5,-1.5)--
(0,-1.5)--(0,-1)--(1,-1)--(1,-0.5)--(2,-0.5)--(2,0)--(1.5,0)--(1.5,1.5)--(0.5,1.5)--(0.5,1)--(2.5,1);

\draw [very thick, big arrow 3] (2.5,1) --(3,1);  
\draw [very thick, big arrow 3] (3,1)--(3,1.5);  
\draw [very thick, big arrow 3] (3,1.5)--(2.5,1.5);  
\draw [very thick, big arrow 3] (2.5,1.5)--(2.5,1);

\draw (4.5,4.5) node[anchor=south west]{$\bb Z^d$};
\draw (0,0) node[anchor=north east]{$0$};
\draw (0,3) node[anchor=south east]{$B_{3k}$};
\draw (0,4.5) node[anchor=south east]{$B_{3(k+1)}$};
\draw (2.7,1.5) node[anchor=south east]{$C_{u_0}$};
\draw (2,1) node[below]{$x$};

\shade[ball color=gray](2.5,1) circle (3pt);

\end{tikzpicture}
\caption{Event $F_k$. After exiting the ball $B_{3k}$, the Ant RW spins forever around a  circuit $C\subset B_{3(k+1)}\backslash B_{3k}$ of length $4$, which is indicated by  arrows. The gray ball represents the root $u_0$ of the circuit $C_{u_0}=(u_0,u_1,u_2,u_3)$. The vertex  $x$ is the last visited vertex of $\p B_{3k}$ before exiting $B_{3k}$.}
\label{Fk}
\end{figure}

We now claim that there exists some $\delta=\delta(d)>0$ such that 
\begin{equation}\label{FE}
\bb P\big(F_k\,|\,E_k\big)\;>\;\delta\,,\quad \forall\,k\in \bb N\,.
\end{equation}
To prove the claim we first apply the strong Markov property at time $\tau_k$, which yields that
\begin{align*}
\bb P\big(F_k\,|\,E_k\big) \;&=\;\bb P\bigg(\tau_k<\infty\,,\,\bigcup_{v\in V(k)} T^{C_{v}}_{\tau_k}\,\bigg|\,\tau_k<\infty\bigg)\\
&=\; \bb E\bigg[\bb P_{\xi_{\tau_k}}\bigg(\bigcup_{v\in V(k)} T^{C_{v}}_{0}\bigg)\,\bigg|\,\tau_k<\infty\bigg]\,.
\end{align*}
Note now that \[\{\tau_k<\infty\}\;=\;\bigcup_{v\in V(k)}\big\{\tau_k<\infty\,,\, X_{\tau_k}=v\big\}\]
and 
\begin{equation}\label{estimate}
\bb P_{\xi_{\tau_k}}\bigg(\bigcup_{v\in V(k)} T^{C_{v}}_{0}\bigg)\;\geq\; \bb P_{\xi_{\tau_k}}\big(T^{C_{u_0}}_{0}\big)\; \text{ on the event }\; \{\tau_k<\infty,X_{\tau_k}=u_0\}\,.
\end{equation}
Therefore, to obtain \eqref{FE} it is enough to get a uniform lower bound for the random variable on the right hand side of~\eqref{estimate} on the event $\{\tau_k<\infty,X_{\tau_k}=u_0\}$.

Immediately after exiting $B_{3k}$, the Ant RW has not crossed any edge in\break $B_{3(k+1)}\backslash B_{3k}$ except the  edge $\{X_{\tau_k-1}, X_{\tau_k}\}$ connecting the vertex $X_{\tau_k-1}\in \p B_{3k}$ to the root $u_0=X_{\tau_k}\in B_{3(k+1)}$ of the circuit  $C_{u_0}=(u_0,u_1,u_2,u_3)$. Therefore, at time $\tau_k$, all edges contained $B_{3(k+1)}\backslash B_{3k}$ have crossing number zero, except the two directed edges $(X_{\tau_k-1},u_0)$ and $(u_0, X_{\tau_k-1})$ which have crossing number $1$ and $-1$ respectively. Thus, it follows that $R_{\tau_k}^{C_{u_0}}\geq -1>-2$.
 Proposition~\ref{lemma:trapped} permits one to obtain the desired $\delta>0$, which is independent of $k$ and hence the claim.

Since $F_k\subset E_{k+1}^\complement$ and by the previous claim, we obtain that
\begin{equation*}
\bb P\big(E_{k+1}^\complement\cap E_k\big)\;>\;\delta\, \bb P\big(E_k\big)\,,\quad \forall\,k\in \bb N\,.
\end{equation*}
Since $E_{k+1}\subset E_{k}$, this implies that
\begin{equation*}
\bb P\big(E_{k}\big)-\bb P\big(E_{k+1}\big)\;>\;\delta\, \bb P\big(E_k\big)\,,\quad \forall\,k\in \bb N\,,
\end{equation*}
which leads to
\begin{equation*}
\bb P\big(E_{k+1}\big)\;<\;(1-\delta)^{k+1}\, \bb P\big(E_0\big)\;=\;(1-\delta)^{k+1}\,,\quad \forall\,k\in \bb N\,.
\end{equation*}
This yields $\bb P \big (\bigcap_{k=1}^\infty E_k\big)=0$ and finishes the proof. 
\end{proof}
\begin{figure}[!htb]
\centering
\begin{tikzpicture}
\draw[step=0.5,lightgray] (-3.5,-3.5) grid (3.5,3.5);
\draw (0,3)--(3,0)--(0,-3)--(-3,0)--cycle;

\draw [very thick] (0,0)--(0.5,0)--(0.5,-1.5)--
(0,-1.5)--(0,-1)--(1,-1)--(1,-0.5)--(2,-0.5)--(2,0)--(1.5,0)--(1.5,0.5)--(1,0.5)--(1,1)--(1.5,1)--(1.5,1.5)--(2.5,1.5)--(2.5,2);

\draw [very thick, dashed] (1.5,1.5)--(0.5,1.5)--(0.5,1)--(0,1)--(0,1.5)--(-0.5,1.5)--(-0.5,0.5);

\draw (3.5,3.5) node[anchor=south west]{$\bb Z^d$};
\draw (0,0) node[anchor=north east]{$0$};
\draw (0,3) node[anchor=south east]{$B_{k}$};

%
\shade[ball color=gray](-0.5,0.5) circle (3pt);
\shade[ball color=black](2.5,2) circle (3pt);

\end{tikzpicture}
\caption{Coupling between the Ant RW on $\bb Z^d$ and the Ant RW's on $B_k\subset \bb Z^d$, $k\in \bb N$. After the hitting time $\sigma_k$ of $\p B_k$, the Ant RW $X_n^{B_k}$ evolves independently of the Ant RW $X_n$. Above, the dashed path represents $X_n^{B_k}$ for times greater than~$\sigma_k$. Note that, immediately after $\sigma_k$, the Ant RW $X_n$ may or may not  exit $B_k$. The gray ball represents the  (final) position of $X_n^{B_k}$ and the black ball the (final) position of $X_n$.}
\label{coupling}
\end{figure}

\begin{proof}[Proof of the Theorem~\texorpdfstring{\ref{thm2.1}}{2.1} in the case \texorpdfstring{$G=\bb Z^d$ with $d\geq 2$}{Zd, d greater than 2}]

For any $k\in\bb N$ we denote by $(X_n^{B_k})_{n\geq 0}$ the Ant RW on $B_k$.  We will construct a  coupling  
\begin{equation*}
\Big((X_n)_{n\geq 0}, (X_n^{B_1})_{n\geq 0}, (X_n^{B_2})_{n\geq 0}, (X_n^{B_3})_{n\geq 0},\ldots\Big)
\end{equation*} 
of all these stochastic processes.   To do so, we first assume that $(X_n)_{n\geq 0}$ has been constructed on some probability space, which will be enriched in the sequel. On this probability space, we define stopping times $\sigma_k$ defined via
\begin{align*} 
\sigma_k\;=\;\min\big\{n\geq 0: X_n \in \p B_k\big\} \,.
\end{align*}
To construct $(X_n^{B_k})_{n\geq 0}$ from $(X_n)_{n\geq 0}$, we let $X_n^{B_k}\overset{\text{def}}{=}X_n$ for  $n< \sigma_k$. If $\sigma_k<\infty$, then for $n\geq \sigma_k$ we let $(X_n^{B_k})_{n \geq\sigma_k}$ evolve independently of $(X_n)_{n\geq \sigma_k}$ on $B_k$, see Figure~\ref{coupling} for an illustration. One then readily checks that $(X_n^{B_k})_{n\geq 0}$ indeed has the law of the Ant RW on $B_k$. Moreover, for $n< \sigma_k$, one has that 
\begin{equation}\label{62}
X_n\;=\; X_n^{B_{k}}\;=\;X_n^{B_{k+1}}\;=\;X_n^{B_{k+2}}\;=\;\cdots
\end{equation}
By Theorem~\ref{thm2.1} for finite graphs, we know that for any $k\in \bb N$ there exists a random  circuit $C=C(k)$ such that 
$(X_n^{B_k})_{n\geq 0}$ is almost surely eventually trapped in  $C$. By Proposition~\ref{prop62}, the Ant RW $X$ on $\bb Z^d$ is bounded. Hence, almost surely there exists a random index $k\geq 1$ such that $\sigma_k=\infty$ and hence \eqref{62} holds for any $n\in \bb N$. Therefore, the Ant RW on $\bb Z^d$ is almost surely trapped in some (random) circuit $C$,  thus concluding the proof.
\end{proof}

\begin{remark}\label{lattice}
\rm The key property of the lattice $\bb Z^d$, $d\geq 2$, in proof of Theorem~\ref{thm2.1} item b) is the presence of circuits of fixed length starting from any vertex (outside of any given large set), which lead to the conditional probability \eqref{FE}. Keeping this in mind, the  proof of Theorem~\ref{thm2.1} item b) can be easily adapted to different lattices as the slab $\{1,\ldots,N\}\times \bb Z^d$, regular non-square lattices etc.
\end{remark}

\appendix

\section{Proof  of \texorpdfstring{Proposition~\ref{th:it_is_markovian}}{Proposition 2.1}}
\label{sec:d=1}

First observe that all crossing numbers $c_n(x,y)$ take values in the set $\{-1,0,$\break $1\}$. Indeed, this is a direct consequence of the fact that the graph $\bb Z$ is  a tree. Thus, for each $n$ and each pair $x,y\in\bb Z$  we have that $a_n(x,y)\in\{e^{-\beta}, 1, e^{\beta}\}$. See Figure~\ref{dim_one} for an illustration. 
\begin{figure}[H]
\centering
\begin{tikzpicture}
\centerarc[thick,<-](1.5,0.3)(10:170:0.45);
\centerarc[thick,->](1.5,-0.3)(-10:-170:0.45);
\centerarc[thick,<-](2.5,0.3)(10:170:0.45);
\centerarc[thick,->](2.5,-0.3)(-10:-170:0.45);
\centerarc[thick,<-](3.5,0.3)(10:170:0.45);
\centerarc[thick,->](3.5,-0.3)(-10:-170:0.45);
\centerarc[thick,->](4.5,-0.3)(-10:-170:0.45);
\centerarc[thick,<-](4.5,0.3)(10:170:0.45);
\centerarc[thick,->](5.5,-0.3)(-10:-170:0.45);
\centerarc[thick,<-](5.5,0.3)(10:170:0.45);
\centerarc[thick,->](6.5,-0.3)(-10:-170:0.45);
\centerarc[thick,<-](6.5,0.3)(10:170:0.45);
\centerarc[thick,->](7.5,-0.3)(-10:-170:0.45);
\centerarc[thick,<-](7.5,0.3)(10:170:0.45);
\centerarc[thick,->](8.5,-0.3)(-10:-170:0.45);
\centerarc[thick,<-](8.5,0.3)(10:170:0.45);

\draw (1,0) -- (9,0);
\draw[dashed] (0,0) -- (1,0);
\draw[dashed] (9,0) -- (10,0);

\shade[ball color=black](7,0) circle (0.25);

\filldraw[fill=white, draw=black]
(1,0) circle (.25)
(2,0) circle (.25)
(3,0) circle (.25)
(4,0) circle (.25)
(5,0) circle (.25)
(6,0) circle (.25)
(8,0) circle (.25)
(9,0) circle (.25);

\draw 
(0.6,-0.1) node[anchor=north] {\tiny $\bf-3$}
(1.6,-0.1) node[anchor=north] {\tiny $\bf-2$}
(2.6,-0.1) node[anchor=north] {\tiny $\bf-1$}
(3.6,-0.1) node[anchor=north] {\tiny $\bf 0$}
(4.6,-0.1) node[anchor=north] {\tiny $\bf 1$}
(5.6,-0.1) node[anchor=north] {\tiny $\bf 2$}
(6.6,-0.1) node[anchor=north] {\tiny $\bf 3$}
(7.6,-0.1) node[anchor=north] {\tiny $\bf 4$};
\draw (1.5,0.8) node[anchor=south]{\small $1$};
\draw (1.5,-0.8) node[anchor=north]{\small $1$};
\draw (2.5,0.8) node[anchor=south]{\small $1$};
\draw (2.5,-0.8) node[anchor=north]{\small $1$};
\draw (3.5,0.8) node[anchor=south]{\small $1$};
\draw (3.5,-0.8) node[anchor=north]{\small $1$};
\draw (4.5,0.8) node[anchor=south]{\small $e^\beta$};
\draw (4.5,-0.8) node[anchor=north]{\small $e^{-\beta}$};
\draw (5.5,0.8) node[anchor=south]{\small $e^\beta$};
\draw (5.5,-0.8) node[anchor=north]{\small $e^{-\beta}$};
\draw (6.5,0.8) node[anchor=south]{\small $e^\beta$};
\draw (6.5,-0.8) node[anchor=north]{\small $e^{-\beta}$};
\draw (7.5,0.8) node[anchor=south]{\small $1$};
\draw (7.5,-0.8) node[anchor=north]{\small $1$};
\draw (8.5,0.8) node[anchor=south]{\small $1$};
\draw (8.5,-0.8) node[anchor=north]{\small $1$};
\end{tikzpicture}
\caption{Illustration of the jump weights, where  $X_n=3$.}
\label{dim_one}
\end{figure}
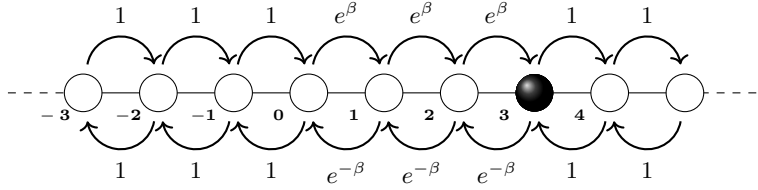
We can say even more about the weights.  
To that end, for an edge $e=(e_-, e_+)$ with starting point $e_-$ and end point $e_+$ in $\bb Z$ and for $x\in\bb Z$, we write $e\leq x$ if $e_-\vee e_+\leq x$ and $e\geq x$ if $e_-\wedge e_+\geq x$. Moreover we also write $a_n(e)$ instead of $a_n(e_-, e_+)$.
We then claim that for all $n\in\bb N$ and all directed edges $e=(e_-, e_+)$, there are three cases.
\begin{itemize}
\item[(1)] $X_n=0$. In this case we simply have $a_n(e)= 1$.
\item[(2)] $X_n>0$. In this case one has
\begin{equation*}
a_n(e)\;=\; 
\begin{cases}
1, &\text{if } e\leq 0 \text{ or }e\geq X_n,\\
e^{\beta}, &\text{if }0\leq e \leq X_n \text{ and also }e_-< e_+,\\
e^{-\beta}, &\text{if }0\leq e\leq X_n \text{ and also }e_-> e_+.
\end{cases}
\end{equation*}
\item[(3)] $X_n <0$. In this case one has
\begin{equation*}
a_n(e)\;=\; 
\begin{cases}
1, &\text{if } e\geq 0 \text{ or }e\leq X_n,\\
e^{\beta}, &\text{if }X_n\leq e \leq 0 \text{ and also }e_-> e_+,\\
e^{-\beta}, &\text{if }X_n\leq e \leq 0 \text{ and also }e_-< e_+.
\end{cases}
\end{equation*}
\end{itemize}
The above is easily proved by induction on $n$. In particular, we see that the position of the walk at time $n$ completely determines the environment $(a_n(\cdot,\cdot))_{n\in\bb N}$ and therefore $(X_n)_{n\in\bb N}$ is a Markov chain.
The transition probabilities claimed in~\eqref{eq:trans1d} are then an immediate consequence of~\eqref{eq:transitions}.

To deduce the desired transience we then define the following sequence of stopping times
\begin{align*}
\tau_1&\;=\;\min\{n> 0: X_n=0\} \; \text{ and}\\
\tau_k&\;=\;\min\{n> \tau_{k-1}: X_n=0\}\,, \; \text{ for }k\geq 2\,.
\end{align*}
Now observe that as a consequence of~\eqref{eq:trans1d} the Ant RW behaves as an asymmetric random walk with bias $\frac{1-e^{-\beta}}{1+e^{-\beta}}$  to the right on the set $\{x\geq 1\}$, respectively to the left  on the set  $\{x\leq -1\}$.
Hence, we see that $(X_{\tau_k+1},\ldots, X_{\tau_{k+1}})$ has the distribution of the asymmetric random walk just described on $x\geq 1$ provided that $X_{\tau_k+1}=1$ respectively on $x\leq -1$ provided that $X_{\tau_k+1}=-1$.  Moreover it follows from~\eqref{eq:trans1d} that for all $k$
\begin{equation*}
\bb P(X_{\tau_k+1}=\pm 1)\; =\; \frac12\,.
\end{equation*}
Furthermore the processes $(Y^k_n)_{0\leq n\leq \tau_{k+1}-\tau_k} = (X_{\tau_k+n})_{0\leq n\leq \tau_{k+1}-\tau_k}$ indexed by $k\in\bb N$ are independent of each other. Hence, in view of the fact that the asymmetric random walk has  positive probability of never returning to the origin, we conclude that almost surely there is $k$ such that $\tau_k=\infty$. Thus, almost surely the walk eventually does not return to the origin. 
Finally, the law of  large numbers \eqref{lln} is a consequence of the law of numbers for the asymmetric random walk.

\section*{Acknowledgements}
The authors would like to thank Milton Jara for helping us with the argument on the proof of Proposition~\ref{prop62}. The authors are also grateful with Augusto Quadros and to the anonymous referees for helping us to simplify the proof of Theorem \ref{thm2.1}.  D.~E. gratefully acknowledges financial support 
from the National Council for Scientific and Technological Development - CNPq via a 
Universal grant 409259/2018-7. T.~F. was supported by a project Jovem Cientista-9922/2015, FAPESB-Brazil and by National Council for Scientific and Technological Development (CNPq) through a Bolsa de Produtividade.  G.~R. was supported by a Capes/PNPD fellowship 888887.313738/2019-00.
D.~E and G.~R are also thankful to the Hausdorff Institute for their hospitality during the program \textit{Randomness, PDEs and Nonlinear Fluctuations} during which part of this work was completed.
G.~R. are also thankful to IMPA for their hospitality and support during the summer school program of 2020.

\bibliographystyle{plain}
\bibliography{bibliografia}

\end{document}